\documentclass{article}
\usepackage{amsmath}
\usepackage{amssymb}
\usepackage{latexsym}
\usepackage{amsthm}
\usepackage{mathrsfs}
\usepackage{wasysym}
\usepackage{fancyhdr}
\usepackage{amsfonts}
\usepackage{amssymb}%Use Miscell. symbols
\usepackage{epsfig}
\usepackage{epstopdf}

\newtheorem{thm}{Theorem}[section]
\newtheorem{cor}[thm]{Corollary}
\newtheorem{lemma}[thm]{Lemma}
\newtheorem{prop}[thm]{Proposition}

\newtheorem{prob}[thm]{Problem}
\begin{document}

\title{Inverses of Bipartite Graphs}

 \author{Yujun Yang\thanks{School of Mathematics and Information Science,
  Yantai University, Yantai, Shandong 264005, China. Partially supported by a grant from National Natural Sciences Foundation of China (No. 11671347).}\; and
  Dong Ye\thanks{Corresponding author. Department of Mathematical Sciences and
  Center for Computational Sciences, Middle Tennessee State University,
  Murfreesboro, TN 37132; Email: dong.ye@mtsu.edu. Partially supported by a grant from Simons Foundation (No. 359516).}
}
\date{ }

\maketitle

\begin{abstract}
Let $G$ be a bipartite graph and its adjacency matrix $\mathbb A$.
If $G$ has a unique perfect matching, then $\mathbb A$ has an inverse $\mathbb A^{-1}$
which is a symmetric integral matrix, and hence the adjacency matrix of a multigraph.
The inverses of bipartite graphs with
unique perfect matchings have a strong connection to M\"obius functions of
posets. In this note, we characterize all bipartite
graphs with a unique perfect matching whose adjacency matrices have
inverses diagonally similar to non-negative matrices, which
settles an open problem of Godsil on inverses of bipartite graphs
in [Godsil, Inverses of Trees, Combinatorica 5 (1985) 33-39].
\end{abstract}

\section{Introduction}
Throughout the paper, a graph means a simple graph (no loops and parallel edges). If
parallel edges and loops are allowed, we use
multigraph instead.
Let $G$ be a bipartite graph with bipartition $(R,C)$. The adjacency matrix  $\mathbb A$ of
$G$ is defined such that the $ij$-entry $(\mathbb A)_{ij}=1$ if $ij\in E(G)$, and 0 otherwise. The bipartite
adjacency matrix $\mathbb B$ of $G$ is defined as the $ij$-entry $(\mathbb B)_{ij}=(\mathbb A)_{ij}=1$
for $i\in R$ and $j\in C$. So $\mathbb B$ is an $|R|\times |C|$-matrix and
\[\mathbb A=\begin{bmatrix} 0 & \mathbb B\\
                                            \mathbb B^\intercal & 0
                                            \end{bmatrix}.\]

A {\em perfect matching} $M$ of $G$ is a set of disjoint edges covering all vertices of $G$.
If a bipartite graph $G$ has a perfect matching, then its bipartite adjacency matrix $\mathbb B$
is a square matrix. Godsil
proved that if a bipartite graph $G$ has a unique perfect matching, then $\mathbb B$ is similar to a lower
triangular matrix with all diagonal entries equal to 1
by permuting rows and columns (\cite{G}, see also \cite{SC}). So in the following, we
always assume that the bipartite adjacency matrix of a bipartite graph with a unique perfect matching is a lower triangular matrix.
Clearly, $\mathbb B$ is invertible and its inverse is an integral matrix (cf. \cite{G, YYMK}).
If $\mathbb B^{-1}$ is non-negative (i.e. all entries are non-negative), then it is the bipartite adjacency matrix of another bipartite multigraph:
the $ij$-entry is the number of edges joining the vertices $i$ and $j$.
However, the adjacency matrix of a graph $G$ has a non-negative inverse if and only if the graph $G$ is the disjoint union of $K_2$'s and $K_1$'s (cf. Lemma 1.1 in \cite{HM}, and \cite{H76}).

%the inverse of the adjacency matrix of a connected graph $G$ is non-negative if and only if
%$G$ is the disjoint union of copies of $K_2$ .

The inverse of $\mathbb B$ is {\em diagonally similar} to a non-negative integral matrix $\mathbb B^{+}$  if there exists a diagonal matrix $\mathbb D$ with -1
and 1 on its diagonal such that $\mathbb D\mathbb B^{-1}\mathbb D= \mathbb B^+$. So $\mathbb B^+$ is a bipartite adjacency matrix of
a bipartite multigraph that is called the inverse of the bipartite graph $G$ in \cite{G} (a broad definition of
graph inverse is given in the next section).
%The inverses
%of graphs have applications in estimations of median eigenvalues of graphs (cf. \cite{G, YYMK}),
%which have physical meanings in Quantum Chemistry.
The following is a problem
raised by Godsil in \cite{G} which is still open \cite{G15}.

\begin{prob}[Godsil, \cite{G}]\label{prob}
Characterize the bipartite graphs with unique perfect matchings such
that $\mathbb B^{-1}$ is diagonally similar to a non-negative matrix.
\end{prob}

\begin{figure}[!hbtp]\refstepcounter{figure}\label{fig:poset}
\begin{center}
\includegraphics[scale=1]{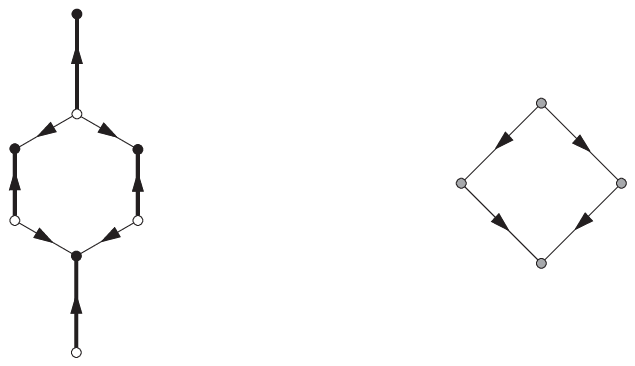}\\
{Figure \ref{fig:poset}: A bipartite graph with a unique perfect matching (left) and its corresponding digraph (right).}
\end{center}
\end{figure}

The bipartite graphs with unique perfect matchings are of particular
interest because of the combinatorial interest of their inverses (cf. \cite{G,MM}).
Let $G$ be a bipartite graph with a unique
perfect matching $M$, and $(R,C)$ be the bipartition of $G$.
Let $D$ be the digraph obtained from $G$ by orienting all edges from $R$ to $C$
and then contracting all edges in $M$.
Simion and Cao proved that the digraph $D$ is acyclic (\cite{SC}). For example,
see Figure~\ref{fig:poset}.
The acyclic digraph $D$
corresponds to a poset $(\mathcal P, \le)$ such that for $a_i,a_j\in\mathcal P=V(D)$,
there is a directed path from $a_j$ to $a_i$ in $D$ if and only if
$a_i\le a_j$ in $\mathcal P$. The {\em Zeta matrix} $\mathbb Z$ of $\mathcal P$ is defined as follows
(cf. Chapter 4 in \cite{MA})
\[(\mathbb Z)_{ij} :=\left\{
 \begin{array}{ll}
1   &\mbox{if $a_i\le a_j $;}\\
0  &\mbox{otherwise,}
 \end{array}
 \right.
\]
The modified Zeta matrix $\mathbb Z(x)$ of $\mathcal P$ is obtained by replacing the entry 1 by
% be the modified Zeta matrix
%of $\mathcal P$
a variable $x$ for a comparable pair of $\mathcal P$
which is not an arc of the digraph $D$. Then $\mathbb Z(1)=\mathbb Z$ the Zeta
matrix of $\mathcal P$, and $\mathbb Z(0)=\mathbb B$ the bipartite adjacency matrix of $G$. Note that $\mathbb Z(0)$ is  the adjacency matrix of $D$  and $\mathbb Z (1)$ is the adjacency matrix of the transitive closure of $D$.
The M\"obius function on the interval of $[a_i,a_j]$
in $\mathcal P$ is $\mu(a_i,a_j)=(\mathbb Z^{-1})_{ij}$ (see Ex. 22 in
Chapter 2 of Lov\'asz on page 216 in \cite{L}), and $\mathbb Z^{-1}$ is the M\"obius matrix
of $(\mathcal P, \le)$.  On the other hand, the Zeta matrix of a poset $(\mathcal P, \le)$
is a lower triangular matrix, corresponding to a bipartite adjacency matrix
of a bipartite graph with a unique perfect matching. This sets up a connection
between inverses of bipartite graphs with unique perfect matchings and M\"obius
functions of posets.

As observed in \cite{G}, if $(\mathcal P,\le)$ is a geometric lattice (a finite
matroid lattice \cite{Welsh}) or the face-lattice of a convex polytope \cite{BG}), then the M\"obius
matrix of $\mathcal P$ is diagonally similar to a non-negative matrix (cf. Corollary 4.34 in \cite{MA}).
%For bipartite graphs with unique perfect matchings,
Godsil \cite{G} proved that if $G$ is a tree
with a perfect matching, then the inverse of its adjacency matrix is
diagonally similar to a non-negative matrix. Further, it has been observed that
if $G$ and $H$ are two bipartite graphs with the property stated
in Problem~\ref{prob}, then the Kronecker product $G\otimes H$
is again a bipartite graph with
the property \cite{G}. The following is a partial solution to Problem~\ref{prob}.

\begin{thm}[Godsil, \cite{G}]\label{thm:Godsil}
Let $G$ be a bipartite graph with a unique perfect matching $M$ such that $G/M$ is bipartite.
Then $\mathbb B^{-1}$
is diagonally similar to a non-negative matrix.
\end{thm}

Godsil's result was generalized to weighted bipartite graphs with unique perfect matchings by Panda and Pati in~\cite{PP}.
In this paper, we provide a solution to Problem~\ref{prob} as follows.

\begin{thm}\label{thm:main}
Let $G$ be a bipartite graph with a unique perfect matching $M$. Then $\mathbb B^{-1}$ is
diagonally similar to a non-negative matrix if and only if $G$ does not contain an odd flower as a subgraph.
\end{thm}

To define odd flower, we need more notation. Let $G$ be a bipartite multigraph with a perfect matching $M$.
A path $P$ of $G$ is $M$-alternating if $E(P)\cap M$ is a perfect matching of $P$.
For two vertices $i$ and $j$ of $G$, let $\tau(i,j)$ be the number of $M$-alternating paths of $G$
joining $i$ and $j$. Further, let $\tau_o(i,j)$ be the number of $M$-alternating paths $P$ of
$G$  joining $i$ and $j$ such that $\big |E(P)\backslash M\big |$ is odd, %\equiv 1\pmod 2$,
and $\tau_e(i,j)$ be the number of
$M$-alternating paths $P$ joining $i$ and $j$ such that $\big |E(P)\backslash M\big |$ is even. %\equiv 0\pmod 2$.
For a subset $S=\{x_1,x_2,...x_k\}$ of $V(G)$, the {\em $M$-span} of $S$ is defined
as a subgraph of $G$ consisting of all $M$-alternating paths joining $x_i$ and $x_j$
for any $x_i,x_j\in S$, denoted by $\mbox{Span}_{M}(S)$.  An $M$-span is called
a {\em flower} if the vertices of $S$ can be ordered such that
$\tau_o(x_i,x_{j})\ne \tau_e(x_i, x_{j})$ if and only if $|i-j|\equiv 1$ (mod $k$).
A flower is {\em odd} if there is an odd number of vertex pairs  $\{x_i, x_j\}$
with $\tau_o(x_i , x_{j} )>\tau_e(x_i ,x_{j} )$. For  example, see Figure~\ref{fig:flower}.  In  Section 4, it will be shown that the existence of an odd flower means simply that the vertices in $S$ induce a cycle with an odd number of negative edges in the inverse of $G$.

\begin{figure}[!hbtp]\refstepcounter{figure}\label{fig:flower}
\begin{center}
\includegraphics[scale=1]{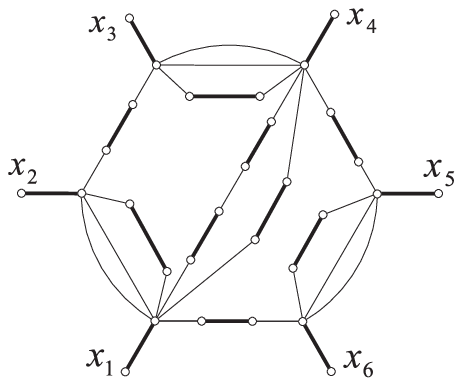}\\
{Figure \ref{fig:flower}: An odd flower: thick edges form a perfect matching $M$.}
\end{center}
\end{figure}

%The following is our main result which provides a solution to Problem~\ref{prob}.

%%%%%%%%%%%%%%%%%%%%%%%%%
\section{Inverses of weighted graphs}
%%%%%%%%%%%%%%%%%%%%%%%%%

A weighted multigraph $(G,w)$ is a multigraph with a weight-function
$w: E(G)\to \mathbb F\backslash\{0\}$ where $\mathbb F$ is a field.
We always assume that a weighted
multigraph has no parallel edges since all parallel edges $e_1,...,e_k$ joining a pair of vertices
$i$ and $j$ can be replaced by one edge $ij$ with weight $w(ij)=\sum w(e_i)$.
The adjacency matrix of a weighted multigraph $(G,w)$, denoted
by $\mathbb A_{w}$, is defined as
\[
\mathbb (\mathbb A_{w})_{ij} :=\left\{
 \begin{array}{ll}
w(ij)    &\mbox{if $i j\in E(G)$;}\\
0  &\mbox{otherwise}
 \end{array}
 \right.
\]
where loops, with $w(ii)\ne 0$, are allowed.
A weighted
multigraph $(G,w)$ is {\em invertible over $\mathbb F$} if its adjacency matrix $\mathbb A_w$ is invertible over $\mathbb F$. Note that $\mathbb A_w$ is a symmetric matrix. Its inverse $\mathbb A_w^{-1}$ is also symmetric and therefore is the adjacency matrix of some weighted graph, which is called the inverse of $(G,w)$. The {\em inverse} of $(G,w)$ is defined as a weighted graph $(G^{-1}, w^{-1})$ whose vertex set is $V(G^{-1})=V(G)$ and whose edge set is $E(G^{-1})=\{ij \, |\,  (\mathbb A_w^{-1})_{ij}\ne 0\}$, and whose weight function is $w^{-1}(ij)=(\mathbb A_w^{-1})_{ij}$. Note that this definition of graph inverse is different from the definitions given in \cite{G} and \cite{MM}.

%A graph can be treated as a weighted graph $(G,w)$ with $\mathbb F=\{0,1\}$, and a signed graph is a weighted graph $(G,w)$ with $\mathbb F=\{-1, 0,1\}$. In the following, assume that $\mathbb F=\mathbb R$.

Let $G$ be a graph. A {\em Sachs subgraph} of $G$ is a spanning subgraph
with only copies of $K_2$ and cycles (including loops) as components. For example,
a perfect matching $M$ of $G$ is a Sachs subgraph.
For convenience, a Sachs subgraph is denoted by $S=\mathcal C\cup M $ where $\mathcal C$
consists of the cycles of $S$ (including loops), and $M$ consists of all components of $S$ isomorphic to $K_2$.
 The following result shows how to compute the determinant of the adjacency matrix of a graph.

\begin{thm}[Harary, \cite{H62}]\label{thm:Harary}
Let $G$ be a graph and $\mathbb A$ be the adjacency matrix of $G$. Then
\[\det(\mathbb A) = \sum_{S} 2^{|\mathcal C|} (-1)^{|\mathcal C|+|E(S)|},\]
where $S=\mathcal C\cup M$ is a Sachs subgraph.
\end{thm}

If $G$ is a bipartite graph with a Sachs subgraph $S=\mathcal C\cup M$, then every cycle $C$
in $\mathcal C$ is of even size and hence its edge set can be decomposed into
two disjoint perfect matchings of $C$. Therefore, $G$ has at least $2^{|\mathcal C|}$ perfect matchings.
So if $G$ is a bipartite graph with a unique perfect matching $M$, then $M$ is the unique Sachs subgraph of $G$. Hence we have the following corollary of the above result, which can also be derived easily from a result of Godsil (Lemma 2.1 in \cite{G}).

\begin{cor}\label{thm:det}
Let $G$ be a bipartite graph with a unique perfect matching $M$. Then
\[\det(\mathbb A) = (-1)^{|M|},\]
where $\mathbb A$ is the adjacency matrix of $G$.
\end{cor}

By Corollary~\ref{thm:det}, the determinant of the adjacency matrix of  a bipartite graph $G$ with a unique perfect matching
is either 1 or $-1$.
So a bipartite graph with a unique perfect matching is always invertible.
The inverse of a graph can be characterized in terms of its Sachs subgraphs as shown in the following theorem, which was originally proved in \cite{YYMK}. However, to make the paper self-contained, we include the proof here as
well.

%The inverse of a bipartite graph with a unique perfect matching can be characterized base on its Sachs subgraph as shown in the following theorem.
%A more general result on inverses of graphs is given in \cite{YYMK}.

% $G$ can be

%For a subgraph $H$ of a weighted graph $(G,w)$. The weight of $H$ is defined by
%\[w(H):=\prod\limits_{e\in E(H)}w(e).\]
%The following result characterizes the inverse of a weighted graph $(G,w)$ based on its Sachs subgraphs.

\begin{thm}[\cite{YYMK}]\label{thm:inverse}
Let $G$ be a graph with adjacency matrix $\mathbb A$, and
\[\mathcal P_{ij}=\{P| P \mbox{ is a path joining } i \mbox{ and } j\ne i \mbox{ such
that } G\backslash V(P) \mbox{ has   a Sachs subgraph } S\}.\]
If $G$ has an inverse $(G^{-1}, w)$, then
\[w(ij)=\left \{
 \begin{array}{ll}
\displaystyle \frac{1}{\det (\mathbb A )}  \sum_{P\in \mathcal P_{ij}} \sum_{S}2^{|\mathcal C|}(-1)^{|\mathcal C|+|E(S)\cup E(P)|} &\mbox{if } i\ne j; \\
\displaystyle \frac{1}{\det (\mathbb A )}\det(\mathbb A_{i,i})&\mbox{otherwise}
\end{array} \right .\]
where $S=\mathcal C\cup M$ is a Sachs subgraph of $G\backslash V(P)$ and $\mathbb A_{i,i}$ is the matrix obtained by deleting $i$-th row and $i$-th column from $\mathbb A$.
\end{thm}
\begin{proof}
Let $G$ be an invertible graph and $(G^{-1},w)$ be its inverse. Assume $G$ has $n$ vertices and $V(G)=\{1,2,...,n\}$.
By the definition of the inverse of a graph, $w(ij)=(\mathbb A^{-1})_{ij}$.

Note that $\mathbb A$ is symmetric and hence $\mathbb A^{-1}$ is also symmetric. By Cramer's rule,
\[(\mathbb A^{-1})_{ij}=(\mathbb A^{-1})_{ji}=\frac{ c_{ij} }{\det(\mathbb A)}\]
where $ c_{ij}=(-1)^{i+j} \det(\mathbb A_{i,j})$ where $\mathbb A_{i,j}$ is the matrix obtained from $\mathbb A$ by deleting $i$-th row and $j$-th column.
Let $\mathbb M_{i,j}$ be the matrix obtained from $\mathbb A$ by replacing  the $(i,j)$-entry
by 1 and all other entries
in the $i$-th row and $j$-th column by 0. Then by the Laplace expansion, $c_{ij}=\det(\mathbb M_{i,j})$

If $i=j$, then $\det(\mathbb M_{i,i})=\det(\mathbb A_{i,i})$. So $\displaystyle w(ii)=(\mathbb A^{-1})_{ii} = \frac{c_{ii}} {\det(\mathbb A)}=\frac{\det(\mathbb M_{i,i})}{\det(\mathbb A)}=\frac{\det(\mathbb A_{i,i})}{\det(\mathbb A)}$. So the theorem holds for $i=j$.
In the following, assume that $i\ne j$.

Let $m_{kl}$ be the $(k,l)$-entry of $\mathbb M_{i,j}$. Recall that the Leibniz formula for the determinant of $\mathbb M_{i,j}$ is
\[\det(\mathbb M_{i,j})=\sum_{\pi \in S_n} \mbox{sgn} (\pi) \prod m_{k\pi (k)},\]
where the sum is computed over all permutations $\pi$ of the set $V(G):=\{1,2,...,n\}$. Since all $(i,l)$-entries ($l\ne j$) of $\mathbb M_{i,j}$
are equal to 0
but the $(i,j)$-entry is 1, only permutations $\pi$ such that $\pi(i)=j$ contribute
to the the determinant of $\mathbb M_{i,j}$.
Let $\Pi_{i\to j}$ be the family of  all permutations on $V(G)=\{1,2,...,n\}$ such that $\pi(i)=j$.
Denote the cycle of $\pi$ permuting $i$ to $j$ by $\pi_{ij}$. For convenience, $\pi_{ij}$ is also
used to denote the set of vertices which corresponds to the elements in the permutation
cycle $\pi_{ij}$, for example, $V(G)\backslash \pi_{ij}$
denotes the set of vertices in $V(G)$ but not in $\pi_{ij}$. Denote the permutation of $\pi$ restricted on $V(G)\backslash \pi_{ij}$ by  $\pi \backslash \pi_{ij}$.
Then
\begin{align*}
\det(\mathbb M_{i, j})
&=\sum_{\pi \in \Pi_{i\to j} } \mbox{sgn}(\pi)\prod_{k\in V(G)\backslash \{i\} } m_{k\pi(k)}\\
&=\sum_{\pi \in \Pi_{i\to j}}\big (\mbox{sgn}(\pi_{ij}) \prod_{k\in  \pi_{ij} \backslash \{i\}} m_{k\pi(k)}\big ) \
\big (\mbox{sgn}(\pi \backslash \pi_{ij})\prod_{k\in V(G)\backslash \pi_{ij}} m_{k\pi(k)} \big ).
\end{align*}
By the definition of $\mathbb M_{i,j}$, if  $k\ne i$ or $l\ne j$, then $m_{kl}=(\mathbb A)_{kl}$, the $(k,l)$-entry of $\mathbb A$.

If the permutation cycle $\pi_{ij}$ does not correspond to a cycle of $G$, then for some $k\in \pi_{ij}$, $k\pi(k)$ is not an edge of $G$ and hence $m_{k\pi(k)}=0$. So $\mbox{sgn}(\pi_{ij})\prod_{k\in \pi_{ij}\backslash \{i\}} m_{k\pi(k)}=0$.
If the permutation cycle $\pi_{ij}$ does correspond to a cycle in the graph $G$, let $P$ be the path from $j$ to $i$ following the permutation order in $\pi_{ij}$. Then  $\mbox{sgn}(\pi_{ij})\prod_{k\in \pi_{ij}\backslash \{i\}} m_{k\pi(k)}=(-1)^{|E(P)|}$.
Note that $\mbox{sgn}(\pi\backslash \pi_{ij})\prod_{k\in V(G)\backslash \pi} m_{k \pi(k)}$ is the determinant of the adjacency matrix of the graph $G\backslash V(P)$. By Theorem~\ref{thm:Harary}, it follows that
\[\mbox{sgn}(\pi\backslash \pi_{ij})\prod_{k\in V(G)\backslash \pi} m_{k \pi(k)}=  \sum_{S} 2^{|\mathcal C|} (-1)^{|\mathcal C|+|E(S)|},\]
where $S=\mathcal C\cup M$ is a Sachs subgraph of $G\backslash V(P)$. For the case that $G\backslash V(P)$ has no Sachs subgraphs, then $\big (\mbox{sgn}(\pi_{ij}) \prod_{k\in  \pi_{ij} \backslash \{i\}} m_{k,\pi(k)}\big ) \
\big (\mbox{sgn}(\pi \backslash \pi_{ij})\prod_{k\in V(G)\backslash \pi_{ij}} m_{k,\pi(k)} \big )=0$.
Hence,
\[
\det(\mathbb M_{i,j})=\sum_{P\in \mathcal P_{ij}} (-1)^{|E(P)|} \big(\sum_{S} 2^{|\mathcal C|}(-1)^{|\mathcal C|+|E(S)|}\big )=\sum_{P\in \mathcal P_{ij}} \sum_{S}  2^{|\mathcal C|}(-1)^{|\mathcal C|+|E(S)\cup E(P)|},
\] where $S=\mathcal C\cup M$ is a Sachs subgraph of $G\backslash V(P)$. The theorem follows immediately from $\displaystyle w(ij)=\frac{\det (\mathbb M_{i,j})}{\det(\mathbb A)}$.
This completes the proof.
\end{proof}

For a bipartite graph $G$ with a unique perfect matching, the weight function of its inverse $(G^{-1},w)$
can be simplified as shown below.

\begin{thm}\label{cor:inverse}
Let $G$ be a bipartite graph with a unique perfect matching $M$, and let
\[\mathcal P_{ij}=\{P| P \mbox{ is an $M$-alternating path joining } i \mbox{ and } j\}.\]
Then $G$ has an inverse $(G^{-1}, w )$ such that
\[w (ij)=\left \{
 \begin{array}{ll}
\displaystyle \sum_{P\in \mathcal P_{ij}}  (-1)^{|E(P)\backslash M|}  &\mbox{if } i\ne j; \\
\displaystyle 0&\mbox{otherwise.}
\end{array} \right .\]
\end{thm}
\begin{proof}
Let $G$ be a bipartite graph with a unique perfect matching $M$. By Corollary~\ref{thm:det},
$G$ has an inverse which is a weighted graph $(G^{-1}, w)$.

For any two vertices $i$ and $j$, let $P$ be a path joining $i$ and $j$. \medskip

{\bf Claim:} {\sl $G\backslash V(P)$ has a Sachs subgraph if and only if $P$ is an $M$-alternating path.} \medskip

{\em Proof of Claim:} If $P$  is an $M$-alternating path, then $G\backslash V(P)$ has a perfect matching.
So $G\backslash V(P)$ has a Sachs subgraph.

Now assume that $G\backslash V(P)$ has a Sachs subgraph.
Note that $G\backslash V(P)$ is a bipartite graph. Every cycle of a Sachs subgraph of $G\backslash V(P)$ is of even size.
So $G\backslash V(P)$ has a perfect matching $M'$.  Therefore,
$P$ is a path with even number of vertices and has a perfect matching $M''$. Hence
$M'\cup M''$ is a perfect matching of $G$. Since $G$ has a unique perfect matching, it follows that $M=M'\cup M''$.
So $P$ is an $M$-alternating path. This completes the proof of Claim.
\medskip

%Let \[\mathcal P_{ij}=\{P\; | \, P \mbox{ is an $M$-alternating path joining } i \mbox{ and } j\}.\]
Let $P$ be a path in  $\mathcal P_{ij}$. Then $G\backslash V(P)$ has a unique perfect matching $M\backslash E(P)$, which is also its
unique Sachs subgraph.
By Claim and Theorem~\ref{thm:inverse}, for $i\ne j$, we have
\[w(ij)=(-1)^{|M|}\sum_{P\in \mathcal P_{ij}} (-1)^{|(M\backslash E(P))\cup E(P)|}=\sum_{P\in \mathcal P_{ij}} (-1)^{|E(P)\backslash M|}.\]
If $i=j$, then $G\backslash \{i\}$ has no perfect matching and hence no Sachs subgraph. By Theorem~\ref{thm:Harary},  $\det (\mathbb A_{i,i})=0$. By Theorem~\ref{thm:inverse}, it follows that
$w(ii)=0$. This completes the proof.
\end{proof}

%From Corollary~\ref{cor:inverse}, a bipartite graph $G$ with a unique perfect matching $M$
%has an inverse, a weighted graph,  whose
%adjacency matrix is not necessarily non-negative.
%Godsil's question relaxes the requirement to be diagonally similar to a non-negative matrix.

%Hence the inverse of $G$ is a weighted graph $(G^{-1},w^{-1})$
%such that $w^{-1}: E(G^{-1})\to \mathbb Z\backslash\{0\}$. For $i, j\in V(G^{-1})$,
%construct a weighted multi-graph $(Q,\sigma)$ such that replace every edge $ij$ by $|w^{-1}(ij)|$
%parallel edges with the same wight $\sigma(ij)\in \{-1,1\}$, the sign of $\sigma(ij)$ is the same as the
%sign of $w^{-1}(ij)$.
%This special family  of weighted graphs $(Q,\sigma)$ are called signed multi-graphs.

%%%%%%%%%%%%%%%%%%%%%
\section{Balanced weighted graphs}
%%%%%%%%%%%%%%%%%%%%%

 Let $(G,w)$ be a weighted graph.
 An edge $ij$ of a weighted graph $(G,w)$ is {\em positive} if $w(ij)>0$ and {\em negative} if  $w(ij)<0$.
A cycle $C$ of $(G,w)$ is {\em negative} if $w(C)=\prod\limits_{ij\in E(C)}w(ij)<0$.
A {\em signed graph} $(G,\sigma)$ is a special weighted graph with a weight function
$\sigma: E(G)\to \{-1,+1\}$, where $\sigma$ is called the {\em signature} of $G$ (see \cite{H53}).
Signed graphs are well-studied combinatorial structures due to their applications in
combinatorics, geometry and matroid theory (cf. \cite{Z93, Z13}).

A {\em switching function} of a  weighted graph $(G,w)$ is a function $\zeta: V(G) \to \{-1,+1\}$,
and the switched weight-function of $w$ defined by $\zeta$ is
$w^{\zeta}(ij):=\zeta(i)w(ij)\zeta(j)$.
Two weight-functions $w_1$ and $w_2$ of a graph $G$
are {\em equivalent} to each other if there exists a switching function
$\zeta$ such that $w_1=w_2^{\zeta}$.
A weighted graph $(G,w)$ is {\em balanced} if there exists a switching function $\zeta$
such that $w^{\zeta}(ij)>0$ for any edge $ij\in E(G)$.
The following is a characterization of balanced signed graphs
obtained by Harary~\cite{H53}.

\begin{prop}[\cite{H53}]
Let $(G,\sigma)$ be a signed graph. Then $(G,\sigma)$ is balanced if and only if $V(G)$ has a bipartition
$V_1$ and $V_2$ such that $E(V_1,V_2)=\{e \ | e\in E(G) \mbox{ and }\sigma(e)=-1\}$.
\end{prop}

For a weighted graph $(G,w)$, define a signed graph $(G,\sigma)$ such that
$\sigma(ij)w(ij)>0$ for any edge $ij\in E(G)$. Then $(G,w)$ is balanced if
and only if $(G,\sigma)$ is balanced. Therefore,
the above result can be easily extended to weighted graphs $(G,w)$ as follows.

\begin{prop}\label{thm:balance}
Let $(G,w)$ be a weighted graph. Then $(G,w)$ is balanced if and only if $V(G)$ has a bipartition
$V_1$ and $V_2$ such that $E(V_1,V_2)=\{e \ | e\in E(G) \mbox{ and } w(e)<0\}$.
\end{prop}

\noindent{\bf Remark.} Let $(G,w)$ be a weighted graph such that $G$ is connected, and
let $E^+:=\{e\ |\ w(e)>0\}$. Let $G/E^+$ be the graph obtained from $G$ by contracting
all edges in $E^+$ and deleting all loops. Then by Theorem~\ref{thm:balance},
$(G,w)$ is balanced if and only if $G/E^+$ is a bipartite multigraph.
Therefore, it takes $O(m)$ steps to determine
whether a weighted graph is balanced or not, where $m$ is the total number of edges of $G$.

\medskip

A direct corollary of the above theorem is the following result.

\begin{cor}\label{cor:negativecycle}
Let $(G,w)$ be a weighted graph. Then $(G,w)$ is balanced if and only if it does
not contain a negative cycle.
\end{cor}

%It follows easily from Theorem~\ref{thm:balance} that a signed multi-graph $(G,\sigma)$ is balanced if and only
%$(G,\sigma)$ has no negative cycles.
%Note that the multiple edges
%joining two same vertices of a signed multi-graph have the same weight.
%The above theorem certainly holds for signed multi-graph $(G, \sigma)$.
%If a  signed multi-graph $(G,\sigma)$ contains two parallel edges $e_1$ and $e_2$ such
%that $\sigma(e_1)\ne \sigma(e_2)$, then $(G,\sigma)$ is not balanced by Theorem~\ref{thm:balance}.
%In the following, we always assume all
%parallel edges of a signed multi-graph $(G,\sigma)$ joining two vertices $i$ and $j$ have
%the same signature, denoted by $\sigma(ij)$.
%Let  $\epsilon_{ij}$ be the
%number of parallel edges of joining $i$ and $j$ in $G$.
%Let $H$ be a simple graph obtained from $G$ by replacing
%all parallel edges by one edge. If $G$ is simple, then $G=H$.
%Define a wight-function $w: E(H)\to \mathbb Z\backslash \{0\}$ such that $w(ij)=\sigma(ij) \epsilon_{ij}$.
%Then a signed multi-graph $(G,\sigma)$ can be
%treated as a weighted graph $(H, w)$ as defined above.
%Define the adjacency matrix of a signed multi-graph $(G,\sigma)$ to be the adjacency
%matrix of the weighted graph $(H,w)$.

Let $(G,w)$ be a weighted graph and  $\mathbb A_w$ be its adjacency matrix.
For a switching function $\zeta: V(G)\to \{-1,+1\}$,
define $\mathbb D_\zeta$
to be a diagonal matrix with $(\mathbb D_\zeta)_{ii}=\zeta(i)$. Then $(G,w_1)$ is
equivalent to $(G,w_2)$ if and only if $\mathbb A_{w_1}=\mathbb D_\zeta
\mathbb A_{w_2}\mathbb D_\zeta$ for some switching function $\zeta$.
So the adjacency matrices of two equivalent weighted graphs are diagonally similar to each other.

\begin{lemma}\label{lem:diag}
Let $G$ be a bipartite graph with a unique perfect matching $M$. Then $\mathbb B^{-1}$ is diagonally similar
to a non-negative matrix if and only if the inverse of $G$ is a balanced weighted graph.
\end{lemma}
\begin{proof}
Since $G$ is invertible, let $(G^{-1},w )$ be the inverse of $G$ by Theorem~\ref{cor:inverse}. Let $\mathbb A$ be the adjacency matrix of $G$ such that
\[\mathbb A=\begin{bmatrix} 0 & \mathbb B\\
                                            \mathbb B^\intercal & 0
                                            \end{bmatrix},\]
                                            where $\mathbb B$ is the bipartite adjacency matrix of $G$, which we assume without loss of generality to be a lower
                                            triangular matrix with 1 on the diagonal.
Then the inverse of $\mathbb A$ is the adjacency matrix of $(G^{-1},w)$
as follows,
\[\mathbb A^{-1}=\begin{bmatrix} 0 & (\mathbb B^\intercal)^{-1}\\
                                            \mathbb B^{-1} & 0
                                            \end{bmatrix}.\]

Note that $\mathbb B^{-1}$ is diagonally similar to a non-negative matrix if and only if
$\mathbb A^{-1}$ is diagonally similar to a non-negative matrix. In other words, if and only if
there exists a diagonal matrix $\mathbb D$ with $(\mathbb D)_{ii}\in \{-1,+1\}$
such that $\mathbb D\mathbb A^{-1}\mathbb D$ is non-negative.
Define a switching function $\zeta: V\to \{-1,+1\}$ such that
$\zeta(i)=(\mathbb D_{ii})$.
Note that
\[w ^{\zeta}(ij)=\zeta(i)w (ij)\zeta(j)=\zeta(i)(\mathbb A^{-1})_{ij}\zeta(j)=(\mathbb D\mathbb A^{-1}\mathbb D)_{ij}.\]
Hence $\mathbb A^{-1}$ is diagonally similar to a non-negative matrix
if and only if there exists a switching function $\zeta$ such that $w ^{\zeta}: E(G^{-1})\to \mathbb R^+$.
Let $V_1=\{v\in V\ | \zeta(v)=1\}$ and $V_2=\{v\in V\ |\zeta(v)=-1\}$.
So the existence of the switching function $\zeta$ is equivalent to the existence of
a bipartition $V_1$ and $V_2$ of $V$ such that $E(V_1, V_2)=\{e \ |w (e)<0\}$.
By Proposition~\ref{thm:balance}, it follows that $\mathbb B^{-1}$ is diagonally similar to a
non-negative matrix if and only if $(G^{-1},w )$ is balanced.
\end{proof}

By Lemma~\ref{lem:diag}, Godsil's problem is equivalent to ask
which bipartite graphs with unique perfect matchings
have a balanced weighted graph as its inverse.

%%%%%%%%%%%%%%%%%%%%%%%%%
\section{Proof of Theorem~\ref{thm:main}}
%%%%%%%%%%%%%%%%%%%%%%%%%

 Now, we are ready to prove our main result.\medskip

 \noindent{\bf Theorem 1.3.} Let $G$ be a bipartite graph with a unique perfect matching $M$. Then $\mathbb B^{-1}$ is
diagonally similar to a non-negative matrix if and only if $G$ does not contain an odd flower as a subgraph.
\medskip

\begin{proof} Let $G$ be a bipartite graph with a
unique perfect matching $M$ and $\mathbb B$ the bipartite adjacency matrix of $G$.
For any two vertices $i$ and $j$ of $G$, let
\[\mathcal P_{ i j}=\{P\ | P \mbox{ is an $M$-alternating path joining } i\mbox{ and } j\}.\]

\noindent $\Rightarrow$:
Assume that $\mathbb B^{-1}$ is diagonally similar to a non-negative matrix. We need to
show that $G$ does not contain an odd flower.
Suppose on the contrary that $G$ does contain a vertex subset $S=\{x_1,...,x_k\}$
such that $\mbox{Span}_M(S)$ is an odd flower.
Then all paths in $\mathcal P_{x_ix_{i+1}}$ belong to $\mbox{Span}_M(S)$.
By Theorem~\ref{cor:inverse},  $G$ has an inverse $(G^{-1},w)$ where,
\[w(x_ix_{i+1})=
\sum_{P\in \mathcal P_{x_i x_{i+1}}}  (-1)^{|E(P)\backslash M|}.\]
So $w(x_ix_{i+1})\in \mathbb Z\backslash \{0\}$ and
$w(x_ix_{i+1})<0$ if and only if $\tau_o(x_i,x_{i+1})>\tau_e(x_i,x_{i+1})$.
Note that  $\mbox{Span}_M(S)$ is an odd flower. So $C=x_1\cdots x_kx_1$ is
a negative cycle in $(G^{-1}, w)$. By Corollary~\ref{cor:negativecycle},
$(G^{-1},w)$ is not balanced.
Hence $\mathbb B^{-1}$ is not diagonally similar to a non-negative matrix by Lemma~\ref{lem:diag},
a contradiction.
 \medskip

\noindent $\Leftarrow$:
Assume that $G$ does not contain an odd flower as a subgraph. We need to show that $\mathbb B^{-1}$ is diagonally similar to a non-negative matrix.
Suppose on the contrary that $\mathbb B^{-1}$ is not diagonally similar to
a non-negative matrix.
Then by Lemma~\ref{lem:diag}, its inverse $(G^{-1},w)$ is not balanced,
and hence contains a negative cycle by Corollary~\ref{cor:negativecycle}.
Choose a shortest negative cycle  $C:=x_kx_1\cdots x_k$ (i.e., $k$ is as small as possible).
%By Theorem~\ref{cor:inverse},
Then $w(x_ix_{i+1})\ne 0$ as $x_ix_{i+1}$ is an edge of $G^{-1}$ (subscripts modulo $k$). Hence $\tau_o(x_i,x_{i+1})\ne \tau_e(x_i,x_{i+1})$ (subscripts modulo $k$).
Let $S=\{x_1,...,x_k\}$. In the following, we are going to prove $\mbox{Span}_M(S)$ is an
odd flower.

Since $C$ is a smallest negative cycle of $(G^{-1},w)$, it
follows that $C$ has no chord, which implies that $\tau_o(x_i, x_j)=\tau_e(x_i,x_j)$ if $x_i$ and $x_j$ are not consecutive on $C$. In other words, $\tau_o(x_i,x_j)\ne \tau_e(x_i,x_j)$ if and only if $|i-j|\equiv 1$ (mod $k$).
Note that $C$ is a negative cycle. So $C$ contains an odd number
of negative edges. Hence, there is an odd number of vertex pairs $\{x_i,x_{i+1}\}$
 such that $\tau_o(x_i,x_j)>\tau_e(x_i,x_j)$.
Hence $\mbox{Span}_M(S)$ is an odd flower, a contradiction. This completes the proof.
\end{proof}

\noindent{\bf Remark.} For a matrix $\mathbb B$, its inverse can be found in $O(n^3)$ steps. Note that it takes
$O(n^2)$ steps to determine whether
the inverse $(G^{-1}, w)$ of $G$ is balanced or not.
Hence, it  can be determined in $O(n^3)$ whether $G$ has a balanced weighted graph as inverse or not.

\section*{Acknowledgement}
The authors would like to thank the anonymous referees for their valuable comments to improve the final version of the
paper.

%%%%%%%%%%%%%%%%%%%%%%%%%

\end{document}